\documentclass[letterpaper, 10 pt, conference]{ieeeconf}
\usepackage[english]{babel}

\usepackage[margin=54pt]{geometry}

\title{\LARGE \bf
Coherence in Synchronizing Power Networks \\ with Distributed Integral Control
} 
 
\author{Martin Andreasson, Emma Tegling, Henrik Sandberg and Karl H. Johansson%
\thanks{M. Andreasson, deceased, was with the School of Electrical Engineering and the ACCESS Linnaeus Centre, KTH Royal Institute of Technology, Stockholm, Sweden. He sadly passed away after the completion of this work.
 E. Tegling, H. Sandberg and K.H. Johansson are with the same organization.
Corresponding author: E. Tegling {(\tt\small tegling@kth.se)}}%
\thanks{This work was supported in part by the European Commission, the Knut and Alice Wallenberg Foundation, by the Swedish Research Council through grants 2014-6282 and 2013-5523, and by the Swedish Foundation for Strategic Research.}  }

\IEEEoverridecommandlockouts

\usepackage{mathrsfs}
\usepackage{amsfonts}
\usepackage{amssymb}
\usepackage{amsmath}
\usepackage{graphicx}
\usepackage{caption}
\usepackage{subcaption}
\usepackage{subfloat}
\usepackage{float}
\usepackage{amsthm}
\usepackage{commath}
\usepackage{tikz}
\usepackage{tkz-graph}
\usepackage{pgfplots}
\usepackage[europeanresistors]{circuitikz}
\usetikzlibrary{arrows}
\usepackage{psfrag}
\usepackage{pgfplots}
\usepackage{todonotes}
\usetikzlibrary{external}
\usepackage{blockdiagram}
\usepackage{caption}
\usepackage{subcaption}
\usepackage{multirow}
\pgfplotscreateplotcyclelist{Power_frequency}{%
	{blue,line width=1pt},
	{red,line width=1pt},
	{magenta,line width=1pt},
	{yellow,line width=1pt},		
	{cyan,line width=1pt},
}

\pgfplotscreateplotcyclelist{Power_cost}{%
	{black,line width=1pt},
	{black,dashed,line width=1pt},
	{black,dotted,line width=1pt},
}

\newtheorem{theorem}{Theorem}

\newtheorem{assumption}{Assumption}

\usepackage{cite}

\DeclareMathOperator*{\diag}{diag}

\DeclareMathOperator*{\E}{\mathbb{E}}

\newcommand{\beq}{\begin{equation}}
\newcommand{\eeq}{\end{equation}}
\newcommand{\bq}{\begin{eqnarray}}
\newcommand{\eq}{\end{eqnarray}}
\newcommand{\bqn}{\begin{eqnarray*}}
\newcommand{\eqn}{\end{eqnarray*}}
\newcommand{\bee}{\begin{enumerate}}
\newcommand{\eee}{\end{enumerate}}

\newcommand{\hn}{$\mathcal{H}_2$ }

\newboolean{showcomments}
\setboolean{showcomments}{true}

\newcommand{\henrik}[1]{\ifthenelse{\boolean{showcomments}}
{\textcolor{blue}{Henrik says: #1}}{}}
\newcommand{\emma}[1]{\ifthenelse{\boolean{showcomments}}
{\textcolor{red}{Emma says: #1}}{}}
\newcommand{\martin}[1]{\ifthenelse{\boolean{showcomments}}
{\textcolor{cyan}{Martin says: #1}}{}}

\begingroup
\makeatletter
\renewcommand{\p@subfigure}{}
\makeatother

\newlength\fheight
\newlength\fwidth

\begin{document}
\maketitle

\begin{abstract}
We consider frequency control of synchronous generator networks and study transient performance under both primary and secondary frequency control. We model random step changes in power loads and evaluate performance in terms of expected deviations from a synchronous frequency over the synchronization transient; what can be thought of as \emph{lack of frequency coherence}. We compare a standard droop control strategy to two secondary proportional integral~(PI) controllers: centralized averaging PI control (CAPI) and distributed averaging PI control (DAPI). We show that the performance of a power system with DAPI control is always superior to that of a CAPI controlled system, which in turn has the same transient performance as standard droop control. Furthermore, for a large class of network graphs, performance scales unfavorably with network size with CAPI and droop control, which is not the case with DAPI control. We discuss optimal tuning of the DAPI controller and describe how inter-nodal alignment of the integral states affects performance. Our results are demonstrated through simulations of the Nordic power grid. 
\end{abstract}

\section{Introduction}
Maintaining the frequency close to its nominal value is one of the key control objectives in alternating current~(AC) power systems. This is traditionally achieved through a hierarchy of control actions, starting with decentralized primary control (droop control), via the secondary control layer with the automatic generation control (AGC) to the tertiary control layer, where an optimal economic dispatch of generators takes place~\cite{machowski2008power}. These control layers operate at different time scales and, traditionally, with different degrees of centralization. In light of ongoing developments in power systems, in particular large-scale integration of distributed generation sources, recent years have seen an increased interest for more flexible, distributed schemes for optimal frequency control. 

On the one hand, 
this problem has been addressed through online optimization techniques that exploit the frequency dynamics of the power network, see e.g. \cite{ Zhao2014Design, Mallada2014Optimal,Zhang2015realtime, Li2014Connecting}. These approaches include load-side frequency control~\cite{Zhao2014Design, Mallada2014Optimal}, as well as adaptations of the AGC to incorporate the optimal dispatch problem \cite{Li2014Connecting}, and are typically based on primal-dual gradient methods. 
Another line of research has studied various integral control strategies, which address the secondary control problem, that is, to adjust the steady-state frequency reached through droop control to the desired setpoint ~\cite{Shafiee2012, SimpsonPorco2013synchronization, Andreasson2013_ecc,  andreasson2014_acc, Dorflerhierarchy2015, Trip2016Internal}. Among the proposed controllers are fully decentralized proportional integral (PI) controllers, which have been shown to suffer from poor robustness properties and typically lead to sub-optimal power injections~\cite{andreasson2014_acc, Dorflerhierarchy2015}. When complemented by a distributed averaging of the integral state, however, these controllers eliminate frequency errors while maintaining optimality properties, in what is often referred to as power sharing 
\cite{SimpsonPorco2013synchronization, Andreasson2013_ecc}. 

In this paper, we consider one such distributed averaging PI (DAPI) controller, and compare it to a centralized averaging PI (CAPI) controller as well as to standard droop control, with respect to their control performance. Specifically, we study a network of synchronous generators modeled by coupled swing equations, and evaluate the systems' transient performance under these control laws, after being subjected to random step disturbances in the power loads. We study performance in terms of expected deviations from a synchronous frequency over the transient, or in other words, the \textit{level of coherence} in the network frequency. 

The concept of coherence has been widely used to study performance of networked control systems, in particular, in order to point at fundamental limitations in large-scale networks~\cite{bamieh2012coherence, SiamiMotee2016}. In the context of power networks, related metrics have been used to characterize transient performance in terms of losses due to non-equilibrium power flows~\cite{Bamieh2013Price, Tegling2015Price, Tegling2016_ACC, Grunberg2016}.
These metrics have also served as performance criteria in inertia-allocation problems~\cite{Poolla2016ACC} and have been used to evaluate the efficiency of primal-dual algorithms for optimal frequency control~\cite{Simpson2016CDC}. In all of these works, the performance metric can be cast as an \hn norm of an input-output system where the output typically captures phase angle deviations. 

The performance metric we propose in this paper differs subtly, but in important respects, from these previous works. Firstly, it captures deviations in frequency, rather than phase angles, and therefore represents a more direct notion of control performance and coherence, given the control objective of maintaining a uniform frequency over the power network.  Secondly, we model step changes in loads, as opposed to impulses or white noise, and the proposed metric can therefore not be directly formulated as an \hn norm. 

Our main result shows, to begin with, that CAPI control has the same transient performance as primary droop control. This can be understood by the fact that the secondary control input in CAPI is equal at all nodes, and therefore, while it shifts the equilibrium, it does not affect how fast it is reached. Furthermore, and more importantly, we show that the performance of DAPI control is always superior to that of CAPI and droop control. In particular, and in contrast to the results in~\cite{Tegling2015Price,Tegling2016_ACC}, their respective control performance scales differently with network size. For instance, in sparse network topologies, this may imply an unbounded growth of expected frequency deviations with CAPI or droop control, while they remain bounded for any number of generators with DAPI control.

The DAPI controller is distinguished from the decentralized PI-controller through the distributed averaging of the integral state, which improves robustness and power sharing properties. In this paper, we discuss optimal tuning of this distributed averaging with respect to performance, and extend a related analysis carried out in~\cite{Tegling2016_ACC}. It turns out that, for highly damped systems, the theoretical performance is optimized by setting the distributed averaging gain to zero, leading to decentralized PI control. For other cases, we show that this optimal gain is relatively small. This indicates a trade-off between transient performance and robustness to measurement bias, which will be explored in future work.

The remainder of this paper is organized as follows. We introduce the problem setup and performance metric in Section~\ref{sec:setup}. In Section~\ref{sec:perfeval}, we calculate the performance metric for the different controllers, and in Section~\ref{sec:optimal_DAPI}, we discuss optimal tuning of the DAPI controller. We present a numerical simulation in Section~\ref{sec:Simulations} and conclude in Section~\ref{sec:disc}.

\section{Problem setup}
\label{sec:setup}
\subsection{Definitions}
\label{subsec:def}
Let $I_n$ denote the identity matrix of dimension $n$. A column vector of dimension $n$ whose elements are all equal to $c$ is denoted $c_n$, and an $n$ by $m$ matrix whose elements are all equal to $c$ is denoted $c_{n\times m}$. Denote by $J_n=\frac{1}{n} 1_{n\times n}$.

\subsection{Droop controlled power system}
\label{subsec:Problem_Droop}
Consider a network of $n$ synchronous generators modeled by a graph $\mathcal{G}=(\mathcal{V}, \mathcal{E})$, where $\mathcal{V} = \{1, \dots, n\}$ is the set of generators, and $\mathcal{E} \subset \mathcal{V} \times \mathcal{V}$ is the set of network lines. We here assume a Kron-reduced network model (see, e.g. \cite{Varaiya1985}), in which constant-impedance loads have been absorbed into the line models in $\mathcal{E}$. The neighbor set of generator~$i$ in $\mathcal{G}$ is denoted~$\mathcal{N}_i$.

Each generator $i$, here referred to as a bus, is assumed to obey the swing equation, as described in \cite{machowski2008power}
\begin{align}
 \label{eq:swing_2}
m_i\ddot{\delta}_i + \! d_i  ( \dot{\delta}_i \! - \!  \omega^\text{ref} )\!   =  \! - \!\! \sum_{j\in \mathcal{N}_i}\!\! k_{ij}\sin(\delta_i \! - \! \delta_j) + \! P^m_i + \! u_i,
\end{align} 
where $\delta_i$ is the voltage phase angle and $\dot{\delta}_i - \omega^\text{ref}=: \omega_i$ is the frequency deviation at bus~$i$, in which $\omega^\text{ref}$ is the nominal frequency (typically 50 Hz or 60 Hz).
The constants $m_i$ and $d_i$ are, respectively, the inertia and damping (droop) coefficients, and $P_i^m$ is the power load at bus~$i$. We denote by $u_i$ the input from the secondary controller, which, for the droop-controlled system is not present, so $u_i^\mathrm{droop}=0, ~\forall t\ge 0, ~i\in \mathcal{V}$. The constants $k_{ij} = |V_i||V_j|b_{ij}$ are edge weights, where $|V_i|$ is the voltage magnitude at bus~$i$, and $b_{ij}$ is the susceptance of the line $(i,j)$.
By linearizing~\eqref{eq:swing_2} around the equilibrium where $\delta_i=\delta_j \;\forall i,j \in \mathcal{V}$, and making use of the shifted frequency~$\omega_i$, we write the linearized swing equation as
\begin{eqnarray}
 \label{eq:swing_lin_2}
m_i\dot{\omega}_i + d_i {\omega}_i = - \!\!\sum_{j\in \mathcal{N}_i} k_{ij}(\delta_i - \delta_j) + P^m_i + u_i.
\end{eqnarray} 
 By defining $\delta = [
\delta_1 \hdots , \delta_n
]^T$, and $\omega = [\omega_1, \dots, \omega_n]^T
$, we may rewrite \eqref{eq:swing_lin_2} in vector form:
\begin{align}
\label{eq:swing_vector_2}
\tag{$\mathcal{S}_\text{droop}$}
\begin{bmatrix}
\dot{\delta} \\ M \dot{\omega}
\end{bmatrix} = \begin{bmatrix}
\phantom{-} 0_{n\times n} & \phantom{-} I_{n} \\
- \mathcal{L}_k & -  D
\end{bmatrix} 
\begin{bmatrix}
{\delta} \\ \omega
\end{bmatrix} +
\begin{bmatrix}
0_{n} \\  P^m
\end{bmatrix} +
\begin{bmatrix}
0_{n} \\  u
\end{bmatrix}
\end{align}
where $M=\diag({m_1}, \hdots , {m_n})$, $D=\diag(d_1, \hdots, d_n)$, and $\mathcal{L}_k$ is the weighted Laplacian matrix whose elements $[\mathcal{L}_k]_{ij} = \sum_{ l \in \mathcal{N}_i} k_{il}  $ if $i = j$ and $-k_{ij}$ if $i \ne j$.
%
The vector of loads is denoted 
 $P^m=[
P^m_1,\hdots, P^m_n
]^T$, and  $u=[
u_i,\hdots, u_n
]^T$ contains the secondary controller inputs. Note that $u = 0_n$ if no secondary controller is present. 
\subsection{CAPI controller}
\label{subsec:Problem_CAPI}
The droop-controlled power system \eqref{eq:swing_vector_2} can be shown to return to a stable operating point. However, since droop control is effectively a proportional control law, the equilibrium will in general \emph{not} be the desired operating point $\omega=0_n$. This can be overcome by designing the secondary control input $u$ to make $\omega=0_n$ the only stable equilibrium of \eqref{eq:swing_vector_2}. 
That can be achieved by adding an averaging integral control term, resulting in a \emph{centralized averaging proportional integral (CAPI)} controller, which was proposed in \cite{Andreasson2014TAC}. 
The CAPI controller requires the control signal $u$ to be computed centrally through integration of the average frequency deviation and, subsequently, to be broadcast to all generators. The controller takes the form:
\begin{align}
\begin{aligned}
u_i^\mathrm{CAPI} &= - z \\
q \dot{z} &= \frac 1n \sum_{i\in\mathcal{V}} \omega_i,
\end{aligned}
\label{eq:power_centralized_integral}
\tag{CAPI}
\end{align}
where $q>0$ is a controller gain. 
By substituting $u_i^{\mathrm{CAPI}}$ for $u_i$ in \eqref{eq:swing_lin_2}, we obtain the following closed-loop dynamics in vector form:
\begin{align}
\label{eq:CAPI_vector}
\tag{$\mathcal{S}_\text{CAPI}$}
\begin{bmatrix}
\dot{\delta} \\ M \dot{\omega} \\ q \dot{z}
\end{bmatrix} = \begin{bmatrix}
\phantom{-} 0_{n\times n} & \phantom{-} I_{n}  & \phantom{-} 0_n\\ 
- \mathcal{L}_k & -  D & -1_n \\
\phantom{-} 0_n^T & \phantom{-} \frac{1}{n} \! 1_n^T &\phantom{-}  0
\end{bmatrix} \!\!
\begin{bmatrix}
{\delta} \\ \omega \\ z
\end{bmatrix} \! + \!
\begin{bmatrix}
0_{n} \\ P^m \\ 0
\end{bmatrix}\!\!. 
\end{align}

We remark that this control design assumes that all generators take equal part in the secondary frequency control. The design can, however, be generalized to arbitrary positive weights for each bus, see~\cite{Dorflerhierarchy2015, DorflerACC2016}.

\subsection{DAPI controller}
\label{subsec:Problem_DAPI}
DAPI controllers have recently been proposed for frequency control of synchronous generator networks as well as for microgrids \cite{SimpsonPorco2013synchronization, Andreasson2013_ecc, Trip2016Internal}. By including a distributed averaging of the integral states instead of centralized averaging, this controller structure eliminates the need for a central control entity and can be implemented in a distributed fashion. 
The DAPI controller takes the form
\begin{equation}
\begin{aligned}
u_i^\mathrm{DAPI} &= - z_i \\
q_i \dot{z}_i &= \omega_i -  \sum_{j \in \mathcal{N}_i^C} c_{ij} (z_i - z_j),
\end{aligned}
\label{eq:power_secondary_DAPI_scalar}
\tag{DAPI}
\end{equation}
where $c_{ij}=c_{ji}>0$ and $q_i>0$ are positive gains and $\mathcal{N}_i^C$ is the set of generators that generator~$i$ can communicate with. The communication graph $\mathcal{G^C} = \{\mathcal{V}, \mathcal{E}^C\}$ is assumed to be undirected and connected, i.e., $j\in \mathcal{N}_i^C$ implies $i\in\mathcal{N}_j^C$ and there exists a path along the communication graph connecting any two generators in $\mathcal{V}$. 

The DAPI controller has been shown to achieve the important property of stationary power sharing, implying that the generated power of all generators is equal at steady state~\cite{SimpsonPorco2013synchronization, Andreasson2013_ecc}. It can also be modified to achieve \emph{weighted} power sharing~\cite{SimpsonPorco2013synchronization}.

By substituting $u_i^{\mathrm{DAPI}}$ for $u_i$ in \eqref{eq:swing_lin_2}, we obtain the closed-loop dynamics in vector form as
\begin{align}
\label{eq:DAPI_vector}
\tag{$\mathcal{S}_\text{DAPI}$}
\begin{bmatrix}
\dot{\delta} \\ M \dot{\omega} \\ Q \dot{z}
\end{bmatrix} = \begin{bmatrix}
\phantom{-} 0_{n \! \times \! n} & \phantom{-} I_{n}  & \phantom{-} 0_{n \! \times \! n}\\ 
- \mathcal{L}_k & - D & - I_n  \\
\phantom{-} 0_{n \! \times \! n} & \phantom{-}  I_n & - \mathcal{L}_c
\end{bmatrix} \!\!
\begin{bmatrix}
{\delta} \\ \omega \\ z
\end{bmatrix} \! + \!
\begin{bmatrix}
0_{n} \\ P^m \\ 0_n
\end{bmatrix} \!\!, 
\end{align}
where $\mathcal{L}_c$ is the weighted graph Laplacian of the communication network, defined in analogy to $\mathcal{L}_k$, and $Q=\diag(q_1, \dots, q_n)$.

\subsection{Decentralized PI controller}
\label{subsec:Problem_DePI}
A special case of the DAPI controller is obtained when $c_{ij}=0 \;\forall (i,j)\in \mathcal{E}$, i.e., there is no averaging of the integral states. In this case the DAPI controller reduces to a \textit{decentralized PI (DePI)} controller. DePI controllers have been proposed for frequency control of power networks \cite{Andreasson2012,Andreasson2014TAC}, however, their implementation requires phasor measurements at every node. Furthermore, the power sharing property that is achieved by the DAPI controller is not in general achievable through DePI control~\cite{andreasson2013lic}. The DePI controller takes the form 
\begin{align}
\begin{aligned}
u_i^{\mathrm{DePI}} &= - z_i \\
q_i \dot{z}_i &= \omega_i, 
\end{aligned}
\label{eq:power_secondary_DePI_scalar}
\tag{DePI}
\end{align}
where $q_i$ are positive controller gains.

\subsection{Performance metric}
We use an integral metric to compare transient performance of the closed-loop systems of Sections \ref{subsec:Problem_Droop}--\ref{subsec:Problem_DAPI}. Let the synchronization performance of an arbitrary system $\mathcal{S}$ be given by 
\begin{align}
\norm{\mathcal{S}}_\text{sync}^2 \! =\E \left\{ \! \int_0^\infty \! \frac 1n \! \sum_{i=1}^n \! \left( \! \omega_i(t) - \frac 1n \! \sum_{j=1}^n \omega_j(t) \! \right)^2 \!\! \text{d} t \! \right\} \!,
\label{eq:performance_metric_scalar}
\end{align}
where the expectation is taken over the power load, which is assumed to be a Gaussian zero-mean random variable with unit covariance, i.e.,  \[P^m  \sim \mathcal{N}(0_n, I_n).\] The initial state is assumed to be the origin. 
The metric~\eqref{eq:performance_metric_scalar} can be interpreted as the expected total deviation from a synchronous frequency, normalized by the number of generators. As such, it characterizes the \emph{level of coherence} in the network frequency over the transient.
This performance metric does not require the synchronous frequency to equal the nominal frequency, or even to be constant. It is therefore well-defined for a large class of frequency controllers. In particular, this metric enables the comparison between the transients of droop-controlled generators and secondary controllers, even though the former will not synchronize at $\omega = 0_n$ in general.
The performance metric~\eqref{eq:performance_metric_scalar} can be written in vector-form as
\begin{align}
\norm{\mathcal{S}}_\text{sync}^2  =\E \left\{ \int_0^\infty  y^T(t) y(t) \text{d} t \right\},
\label{eq:performance_metric_vector}
\end{align}
where 
\begin{align}
\label{eq:generic_output}
y(t)= \frac{1}{\sqrt{n}}(I_n-J_n)\omega(t).
\end{align}
We will refer to the metric~\eqref{eq:performance_metric_scalar} as a \emph{synchronization norm} and note that it is related to, but distinct from a standard \hn norm metric. Specifically, the integral~\eqref{eq:performance_metric_scalar} would correspond to an \hn norm of a system with output~\eqref{eq:generic_output} if the input were either white process noise or impulses at time $t = 0$, see~\cite{Tegling2015Price} for details. We argue that the random step disturbance in the load $P_m$ that we consider here is a more relevant input  in the power system context. We remark, however, that this performance analysis will characterize the system performance under ``normal'' operation, i.e. small step disturbances, 
but does not necessarily give insights to the system's robustness towards extreme transient stability events.

The proposed metric~\eqref{eq:performance_metric_scalar} is a suitable indicator of the transient performance of a power network, as it relates directly to the control objective of frequency synchronization. A large synchronization norm~\eqref{eq:performance_metric_scalar} indicates lack of frequency coherence, i.e., non-fulfillment of the control objective. 
 Defining a corresponding metric based on phase deviations, as suggested in previous work \cite{Tegling2015Price,Poolla2016ACC, Grunberg2016}, does not capture this notion of frequency coherence.


\section{Evaluating performance }
\label{sec:perfeval}
In this section, we compute the performance metric \eqref{eq:performance_metric_scalar} for the closed-loop power system controlled with, respectively, a droop controller, a CAPI controller, and a DAPI controller. We first make the following assumptions that will be assumed to hold henceforth, in order to obtain closed-form expressions for the performance metric. 
\begin{assumption}[Uniform system parameters]
\label{ass:uniform_gains}
The inertia constants $m_i$, the damping constants $d_i$, as well as the controller gains $q_i$ are uniform across the network and given by $m$, $d$ and $q$, respectively. 
\end{assumption}

\begin{assumption}[Communication topology]
\label{ass:communication_graph}
The communication graph used in the DAPI controller resembles that of the power network. That is, $c_{ij} = \gamma k_{ij}$ for all $(i,j) \in \mathcal{E} = \mathcal{E}^C$, or equivalently $\mathcal{L}_c = \gamma \mathcal{L}_k$. The parameter $\gamma$ can be thought of as a communication gain that governs the speed of the distributed averaging filter of the DAPI controller. 
\end{assumption}
%

The following theorem summarizes the main result of this section.
\begin{theorem}[Performance calculation]
\label{th:cost_dapi_droop_capi}
Let Assumptions~\ref{ass:uniform_gains} and \ref{ass:communication_graph} hold. 
The performance of the droop-controlled power system \eqref{eq:swing_vector_2} and the CAPI controlled power system \eqref{eq:CAPI_vector} is equal, and given by
\begin{align}
\label{eq:performance_droop}
\boxed{
\norm{\mathcal{S}_\text{droop}}_\text{sync}^2 = \norm{\mathcal{S}_\text{CAPI}}_\text{sync}^2 = \frac{1}{2n} \sum_{i=2}^n \frac{1}{\lambda_i d}. }
\end{align}
The performance of the DAPI controlled power system \eqref{eq:DAPI_vector} is given by
\begin{align}
\label{eq:performance_DAPI}
\boxed{
\norm{\mathcal{S}_\text{DAPI}}_\text{sync}^2 = \frac{1}{2n} \sum_{i=2}^n \frac{1}{\lambda_i d + \frac{dq + \gamma \lambda_i m }{\gamma d q + q^2 + \gamma^2 \lambda_i m} }.
}
\end{align}
\end{theorem}

\begin{proof} 
Consider first the droop controlled power system \eqref{eq:swing_vector_2} with load input $P^m$ and output given by \eqref{eq:generic_output}. Similar to~\cite{Bamieh2013Price}, we begin by performing the state transformations: $\bar{\delta} = U^* \delta$, $\bar{\omega} = U^* \omega$,
where $U$ is the unitary matrix that diagonalizes $\mathcal{L}_k$, i.e., $U^* \mathcal{L}_k U = \Lambda$, in which $\Lambda = \diag(\lambda_1, \dots, \lambda_n)$ is a diagonal matrix of the eigenvalues of $\mathcal{L}_k$. Let the first column of $U$, $\mathbf{u}_1$, be the eigenvector corresponding to $\lambda_1 = 0$, so $\mathbf{u}_1 = \frac{1}{\sqrt{n}}1_n$. It is easily verified that $\norm{\cdot}_\text{sync}$ is invariant under this unitary transformation. In the new coordinates and by Assumption~\ref{ass:uniform_gains}, the dynamics of the droop-controlled power system become
\begin{equation}
\begin{aligned}
\begin{bmatrix}
\dot{\bar{\delta}} \\ m \dot{\bar{\omega}}
\end{bmatrix} &= \begin{bmatrix}
\phantom{-} 0_{n\times n} &\phantom{-}  b I_{n} \\
- \Lambda & - d I_n
\end{bmatrix} 
\begin{bmatrix}
\bar{\delta} \\ \bar{\omega}
\end{bmatrix} +
\begin{bmatrix}
0_{n} \\  I_n
\end{bmatrix} \bar{P}^m \\
\bar{y} &= 
\frac{1}{\sqrt{n}}
\begin{bmatrix}
0 & 0_{n-1}^T\\
0_{n-1} & I_{n-1}
\end{bmatrix}
 \bar{\omega}
 \end{aligned}
\label{eq:swing_vector_2_bar}
\tag{$\bar{\mathcal{S}}_\text{droop}$}
\end{equation}
where $\bar{P}^m = U^*P_m$. Since $\E \{ \bar{P}^m \} = U^* \E \{  P^m\} = 0_n$ and $\E \{ \bar{P}^{m*} \bar{P}^m \} = \E \{ P^{m*}UU^*P^m \} = \E \{ P^{m*} P^m \} = I_n$, we conclude that $\bar{P}^m\sim \mathcal{N}(0_n, I_n)$. 

The system \eqref{eq:swing_vector_2_bar} consists of $n$ decoupled subsystems, meaning that $||\mathcal{S}_\text{droop}||^2_\text{sync} =||\bar{\mathcal{S}}_\text{droop}||^2_\text{sync} =  \sum_{i=1}^n ||\bar{\mathcal{S}}^i_\text{droop}||^2_{\text{sync}}$, where ${\bar{\mathcal{S}}^1_\text{droop}}$ is given by
\begin{equation}
\begin{bmatrix}
\dot{\bar{\delta}}_1 \\ \dot{\bar{\omega}}_1
\end{bmatrix}
= \underbrace{\begin{bmatrix}
0 & 1 \\
0& -\frac{d}{m} 
\end{bmatrix}}_{\triangleq \bar{A}^1_\text{droop}}
\begin{bmatrix}
{\bar{\delta}}_1 \\ {\bar{\omega}}_1 
\end{bmatrix}
+ \begin{bmatrix}
0 \\ \frac 1m
\end{bmatrix} \bar{P}^m_1 ,~~\bar{y}_1 = 0, 
\label{eq:droop_bar_1}
\tag{$\bar{\mathcal{S}}_\text{droop}^1$}
\end{equation}
and
\begin{equation}
\begin{aligned}
\begin{bmatrix}
\dot{\bar{\delta}}_i \\ \dot{\bar{\omega}}_i
\end{bmatrix}
&= \underbrace{\begin{bmatrix}
\phantom{-} 0 & 1 \\
-\frac{\lambda_i}{m} & -\frac{d}{m} 
\end{bmatrix}}_{\triangleq \bar{A}^i_\text{droop}}
\begin{bmatrix}
{\bar{\delta}}_i \\ {\bar{\omega}}_i
\end{bmatrix}
+ \begin{bmatrix}
0 \\ \frac 1m
\end{bmatrix} \bar{P}^m_i  \\
\bar{y}_i &= 
\underbrace{\frac{1}{\sqrt{n}}\begin{bmatrix}
0 & 1
\end{bmatrix}}_{\triangleq \bar{C}^i_\text{droop}}
\begin{bmatrix}
\bar{\delta}_i \\ \bar{\omega}_i
\end{bmatrix}, 
\end{aligned}
\label{eq:droop_bar_i}
\tag{$\bar{\mathcal{S}}_\text{droop}^i$}
\end{equation}
for $i\ge 2$. Since $\bar{y}_1 \equiv 0$, we have that $||\bar{\mathcal{S}}_\text{droop}||^2_\text{sync} =  \sum_{i=2}^n ||\bar{\mathcal{S}}^i_\text{droop}||^2_{\text{sync}}$. 
The equilibrium of \eqref{eq:droop_bar_i} is verified to be 
\begin{align}
\label{eq:droop_i_equilibrium}
\begin{bmatrix}
{\bar{\delta}}_i^0 \\ {\bar{\omega}}_i^0
\end{bmatrix}
&=
\begin{bmatrix}
\frac{1}{\lambda_i} \\ 0
\end{bmatrix}
\bar{P}^m_i \triangleq \bar{B}^i_\text{droop} \bar{P}^m_i.
\end{align}
It follows that the synchronization norm $||\bar{\mathcal{S}}^i_\text{droop}||^2_{\text{sync}}$ can be equivalently obtained by instead considering the dynamics in the translated states
\begin{align*}
\begin{bmatrix}
{\bar{\delta}}_i' \\ {\bar{\omega}}_i'
\end{bmatrix}
&=
\begin{bmatrix}
{\bar{\delta}}_i \\ {\bar{\omega}}_i
\end{bmatrix}
- 
\begin{bmatrix}
{\bar{\delta}}_i^0 \\ {\bar{\omega}}_i^0
\end{bmatrix},
\end{align*}
with initial condition given by \eqref{eq:droop_i_equilibrium}. We recognize that $||\bar{\mathcal{S}}^i_\text{droop}||^2_{\text{sync}}$ is then given by the squared $\mathcal{H}_2$-norm of the translated system, whose transfer function is $\bar{C}^i_\text{droop} (sI_n - \bar{A}^i_\text{droop})^{-1} \bar{B}^i_\text{droop}$. This $\mathcal{H}_2$-norm can be calculated by, e.g., solving the Lyapunov equation for $P_i$
\begin{align*}
\bar{A}^{iT}_\text{droop} P_i + P_i \bar{A}^{i}_\text{droop}= - \bar{C}^{iT}_\text{droop} \bar{C}^i_\text{droop},
\end{align*}
after which the norm is given as $||\mathcal{S}^i_\text{droop}||^2_{\text{sync}} = \text{tr}(\bar{B}^{iT}_\text{droop} P_i \bar{B}^i_\text{droop}) $. In our case, $||\bar{\mathcal{S}}^i_\text{droop}||^2_{\text{sync}} = \frac{1}{2n\lambda_i d}$, and summing over $i=2, \dots, n$ yields the expression~\eqref{eq:performance_droop}. The synchronization norm~\eqref{eq:performance_DAPI} for the DAPI controlled system~\eqref{eq:DAPI_vector} is obtained through analogous calculations. 

Now consider the CAPI controlled system \eqref{eq:CAPI_vector}. By similar calculations as before, we transform the state to obtain $n$ decoupled subsystems. These are given by
\begin{equation}
\begin{bmatrix}
\dot{\bar{\delta}}_1 \\ \dot{\bar{\omega}}_1 \\ \dot{z}
\end{bmatrix} \!
= \! \underbrace{\begin{bmatrix}
0 & 1 & \phantom{-}0 \\
0& \frac{d}{m} & -\frac{1}{m} \\
0 & \frac 1q & \phantom{-}0
\end{bmatrix}}_{\triangleq \bar{A}_\text{CAPI}^1} \!\!
\begin{bmatrix}
{\bar{\delta}}_1 \\ {\bar{\omega}}_1 \\ z
\end{bmatrix}
\! + \! \begin{bmatrix}
0 \\ \frac 1m \\ 0
\end{bmatrix} \! \bar{P}^m_1,~~  \bar{y}_1 = 0,
 \label{eq:CAPI_bar_1}
\tag{$\bar{\mathcal{S}}_\text{CAPI}^1$}
\end{equation}
since $\mathbf{u}_1$ is parallel to $1_n$, and by
\begin{equation} 
\begin{bmatrix}
\dot{\bar{\delta}}_i \\ \dot{\bar{\omega}}_i
\end{bmatrix} \!= \! \underbrace{\begin{bmatrix}
\phantom{-} 0 & 1 \\
-\frac{\lambda_i}{m} & \frac{d}{m} 
\end{bmatrix}}_{\triangleq \bar{A}_\text{CAPI}^i}  \!\!
\begin{bmatrix}
{\bar{\delta}}_i \\ {\bar{\omega}}_i
\end{bmatrix}
\! +\! \begin{bmatrix}
0 \\ \frac 1m
\end{bmatrix} \! \bar{P}^m_i,~~\bar{y}_i = 
\underbrace{\frac{1}{\sqrt{n}} \! \begin{bmatrix}
0 & 1
\end{bmatrix}}_{\triangleq \bar{C}_\text{CAPI}^i} \!\!
\begin{bmatrix}
\bar{\delta}_i \\ \bar{\omega}_i
\end{bmatrix}, 
\label{eq:CAPI_bar_i}
\tag{$\bar{\mathcal{S}}_\text{CAPI}^i$}
\end{equation}
for $i\ge 2$. Since \eqref{eq:droop_bar_1} is unobservable and the systems \eqref{eq:droop_bar_i} and \eqref{eq:CAPI_bar_i} are identical for $i\ge 2$, it follows that $\norm{\mathcal{S}_\text{droop}}_\text{sync}^2 = \norm{\mathcal{S}_\text{CAPI}}_\text{sync}^2$. 
\end{proof}

 It is clear from Theorem~\ref{th:cost_dapi_droop_capi} that $\norm{\mathcal{S}_\text{DAPI}}^2_{\text{sync}} < \norm{\mathcal{S}_\text{droop}}^2_{\text{sync}} $ if the damping $d$ is the same for both systems. We conclude that the synchronization performance of the DAPI controller is always superior to that of the droop and CAPI controllers. Perhaps surprisingly, the CAPI controller does not offer a benefit compared to the droop controller in terms of the performance metric~\eqref{eq:performance_metric_scalar}. This can be understood by noting that CAPI only introduces integral feedback control on the average state, and affects all buses equally. It therefore does not affect the synchronization itself, but rather shifts the equilibrium point. 

\subsection{Implications for large-scale networks}
The results in Theorem~\ref{th:cost_dapi_droop_capi} have implications for how well the control laws scale to large networks. In particular, it is easily shown that $||\mathcal{S}^i_\text{DAPI}||^2_{\text{sync}} < \frac{\gamma d +q}{2d}$ for any underlying network graph $\mathcal{G}$. That is, for any finite $\gamma$, the synchronization norm for DAPI remains bounded, even as the number of generators $n$ increases. For CAPI and droop control, on the other hand, $||\mathcal{S}_\text{droop}||_\text{sync}^2 = ||\mathcal{S}_\text{CAPI}||_\text{sync}^2 \sim \frac{1}{n}\sum_{i = 2}^n \frac{1}{\lambda_i}$, an expression that in many cases scales unfavorably in $n$. 

The expression~$\sum_{i = 2}^n \frac{1}{\lambda_i}$ appears often in the coherency literature, and is related to the concept of total effective resistance in resistor networks. As such, a number of recent studies have focused on characterizing bounds and scalings for this expression in various types of networks, see~\cite{bamieh2012coherence, SiamiMotee2016, Barooah}. While the interested reader is referred to these works for details, we note here that their results imply that $||\mathcal{S}_\text{droop}||_\text{sync}^2 = ||\mathcal{S}_\text{CAPI}||_\text{sync}^2$ grow unboundedly with the network size~$n$ for certain sparse network graphs. Examples of such graphs are tree graphs or graphs that can be embedded in two-dimensional lattices~\cite{SiamiMotee2016,Barooah}.

The unfavorable scaling of the droop and CAPI controllers' performance implies an increasing lack of coherence as the number of generators increases. Therefore, these control laws are not as scalable to large networks as DAPI control, unless the underlying graphs are densely interconnected.


\section{Optimizing DAPI performance}
\label{sec:optimal_DAPI}
In the previous section, we showed that the DAPI controller outperforms the droop and the CAPI controllers in terms of transient synchronization performance. We now focus on optimizing the performance of the DAPI controller by tuning its parameters. By Assumptions~\ref{ass:uniform_gains} and \ref{ass:communication_graph}, this reduces to optimizing $\norm{\mathcal{S}_\text{DAPI}}^2_{\text{sync}} $ over (positive) $\gamma$ and $q$, assuming that the inertia $m$ and the damping $d$ are fixed. 

We begin this discussion by considering the function 
\begin{align}
\label{eq:fdef}
f_i(\gamma,q) \triangleq \frac{dq + m  \lambda_i \gamma}{ d q \gamma + q^2 + m \lambda_i \gamma ^2 }
\end{align}
from the denominator of~\eqref{eq:performance_DAPI}. We note that $\lim_{q\rightarrow 0} f_i(0,q) = \infty$, implying that for a communication gain $\gamma=0$ and an inverse integral gain $q \rightarrow 0$, 
\[\norm{\mathcal{S}_\text{DAPI}}^2_{\text{sync}} = 0.\]
Similarly $\lim_{\gamma \rightarrow \infty} f_i(\gamma,q) = 0$. Thus for any fixed $q$ and for $\gamma \rightarrow \infty$,
\begin{align*}
\norm{\mathcal{S}_\text{DAPI}}^2_{\text{sync}} = \norm{\mathcal{S}_\text{CAPI}}^2_{\text{sync}} = \sum_{i=2}^n \frac{1}{2\lambda_i d},
\end{align*}
that is, the performance of the DAPI controller approaches that of the CAPI (or, equivalently, droop) controller. This is expected, as increasing the communication gain $\gamma$ in the DAPI controller to infinity implies arbitrarily fast distributed averaging, ultimately resembling centralized averaging like in CAPI control. 
On the other hand, setting $q=0$, which can achieve zero synchronization cost, is not practically feasible as it would require infinite control effort. A question of practical relevance is therefore that of choosing the optimal $\gamma$ given a fixed $q$. To address that question, we begin by focusing on a special case.


\subsection{Special case: complete graphs}
\label{subsec:optimality_complete}
We first consider the case in which the power network can be modeled as a complete graph. We further assume that the edge weights $k_{ij}$ for all $(i,j) \in \mathcal{E} $ are uniform and equal to $b$. For this graph, the Laplacian eigenvalues satisfy $\lambda_2 = \dots = \lambda_n = {b n} $. Therefore, minimizing $\norm{\mathcal{S}_\text{DAPI}}^2_{\text{sync}}$ is equivalent to maximizing $f_i(\gamma,q)$ from~\eqref{eq:fdef}, with $\lambda_i =b n$. In this case, we can derive a closed-form expression for the optimal value of $\gamma$, denoted $\gamma^\star$, depending on the network parameters and the gain $q$:

\begin{theorem}[Optimizing performance in complete graphs]
\label{th_optimal_gamma_k}
Let Assumptions~\ref{ass:uniform_gains} and \ref{ass:communication_graph} hold. Assume furthermore that the power network is a complete graph and that the edge weights $k_{ij}$ are uniform and equal to $b$. Let the integral controller gain $q$ be fixed. Then, if $d \ge \sqrt{mb n}$, the value of $\gamma$ that minimizes $\norm{\mathcal{S}_\text{DAPI}}^2_{\text{sync}}$ is $\gamma^\star=0$. Otherwise the optimal choice of $\gamma$ is
\[\gamma^\star = \frac{\sqrt{mb n}q -dq}{mb n}.\]
%
\end{theorem}

\begin{proof}
We rewrite $f_i(\gamma,q) $ as
\begin{align*}
f_i(\gamma,q) = \frac{1}{\gamma + \frac{q^2}{dq+m\lambda_i \gamma}},
\end{align*}
and note that maximizing $f_i(\gamma,q) $ is equivalent to minimizing 
\begin{align*}
g_i(\gamma,q) \triangleq \gamma + \frac{q^2}{dq+m\lambda_i \gamma}.
\end{align*}

Differentiating $g_i(\gamma,q) $ with respect to $\gamma$ and setting the derivative equal to zero yields
\(
 \frac{\partial g_i(\gamma,q)}{\partial \gamma}  = 
 1 - \frac{m\lambda_i q^2}{(dq+m\lambda_i \gamma)^2} =0,
\)
which implies
\begin{align*}
\gamma^\star = \frac{-dq \pm \sqrt{m\lambda_i} q}{m\lambda_i}.
\end{align*}
From the above equation we conclude that there is a positive extreme point $\gamma^\star$ if and only if $d\le \sqrt{m\lambda_i}$. Furthermore, $g_i(0,q) = \frac{q}{d}$ and $\lim_{\gamma \rightarrow \infty} g_i(\gamma,q) = \infty$, so clearly the positive extreme point is a local minimum. It follows that the positive extreme point minimizes $g_i(\gamma,q)$ if $d\le \sqrt{m\lambda_i}$, and otherwise $g_i(\gamma,q)$ is minimized by the lower end point of the interval, namely $\gamma = 0$. Recalling that $\lambda_i= b n$ for $i\ge 2$ concludes the proof. 
\end{proof}


Theorem \ref{th_optimal_gamma_k} answers the question of when DAPI control has a theoretical performance superior to DePI control.
If the damping is sufficiently large, the optimal value of $\gamma$ is zero and thus the DePI controller minimizes the synchronization norm. If, on the other hand, the damping is sufficiently small, the optimal value of $\gamma$ is nonzero. 

We also note that the optimal $\gamma^\star$ increases with increasing $q$ and with decreasing $d$, that is, with decreasing integral gain or damping. An intuitive explanation for this can be given as follows. A smaller damping or integral gain will cause disturbances to spread across the network rather than to be suppressed locally. A DAPI controller with large communication gain $\gamma$ anticipates this spreading by faster distributing the integral control action across the generators. 

Further, it can be shown that, given a fixed $d$ and $q$, the largest possible value for $\gamma^\star$ is attained when $m b n = 4d^2$, and is $\gamma^\star_{\mathrm{max}} = \frac{q}{2\sqrt{m b n}} = \frac{q}{4d}$. Therefore, if the number of generators $n$ grows large, or if $m$ or $b$ are large, $\gamma^\star$ will approach zero. On the other hand, when $n$, $b$ and $m$ are instead small, so that $\sqrt{mb n}\le d$,  $\gamma^\star$ is also zero, as seen from Theorem~\ref{th_optimal_gamma_k}. In conclusion, the averaging taking place within the DAPI controller is of most importance for medium-sized power grids with moderate susceptances and inertia.
%

We note that a result similar to that of Theorem~\ref{th_optimal_gamma_k} was proposed in~\cite{Tegling2016_ACC} for a different performance metric and input scenario. The optimal value for $\gamma$ in the complete graph case is the same in both cases. This shows that choosing the parameter $\gamma$ as $\gamma^\star$ is relevant with respect to several performance metrics.

It is, however, important to remember that the results of Theorem~\ref{th_optimal_gamma_k} do not take the control cost into account. In particular, the DAPI controller is, thanks to its power sharing properties, known to minimize the static generation cost, whereas the DePI controller does not possess this desirable property. The optimality of the DePI controller with respect to minimizing the synchronization norm should thus be seen in light of the DAPI controller's ability to minimize generation costs.


\begin{figure*}[t!]
	\centering
\begin{subfigure}[b]{0.32\textwidth}
\centering
\includegraphics[scale=1]{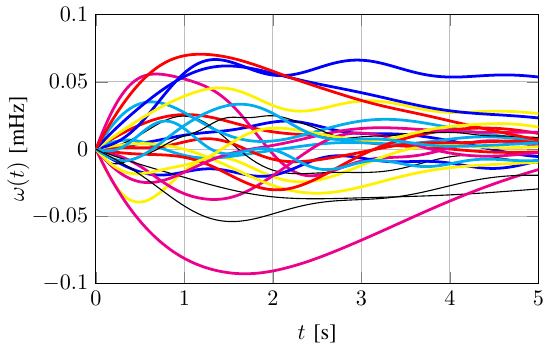}
	\caption{Droop controller.}
	\label{subfig:frequency_nordic_droop}
\end{subfigure}
\;
\begin{subfigure}[b]{0.32\textwidth}
\includegraphics[scale=1]{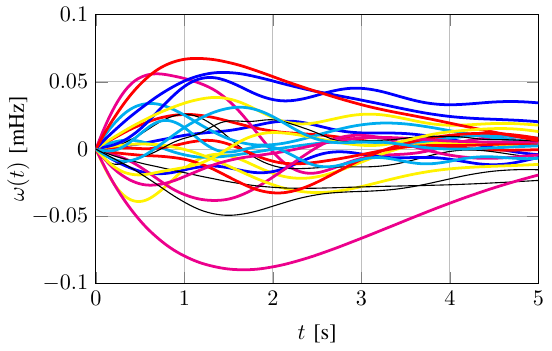}
	\caption{DAPI controller, $\gamma=1$.}
	\label{frequency_nordic_DAPI_large_gamma}
\end{subfigure}
\;
\begin{subfigure}[b]{0.32\textwidth}
\includegraphics[scale=1]{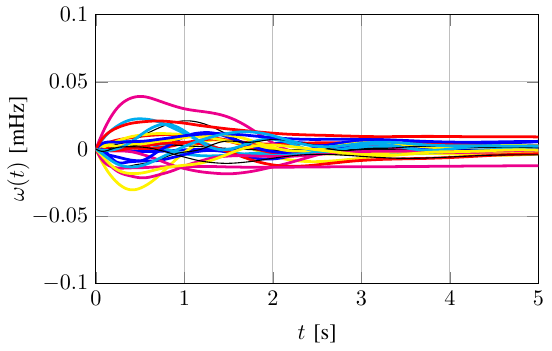}
	\caption{DAPI controller, $\gamma=0.01$.}
	\label{frequency_nordic_DAPI}
\end{subfigure}
\caption{Frequency deviations of a subset of the generators in simulation of the Nordic power grid after a step increase in load at time $t=0$. The figures show the frequencies under (a) droop control, and under DAPI control with (b) $\gamma=1$ and (c) $\gamma=0.01$. DAPI control significantly improves transient performance, provided the parameter $\gamma$ is sufficiently small. }
	\label{fig:frequency_Nordic}
\end{figure*}

\begin{figure}[h]
\centering
\includegraphics[scale=1]{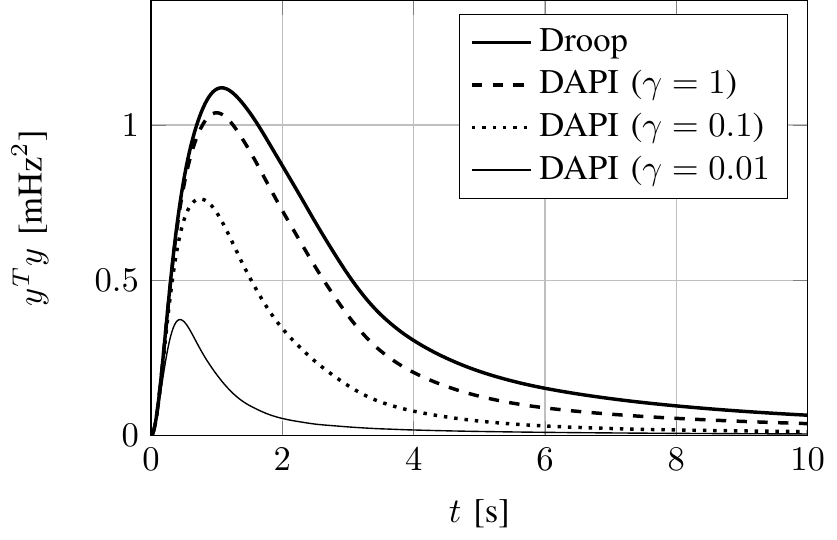}
	
\caption{Squared control error $y^Ty$, where $y(t)= \frac{1}{\sqrt{n}}(I_n-J_n)\omega(t)$, corresponding to the simulations in Fig.~\ref{fig:cost_Nordic_DAPI_droop}. 
The synchronization norm~\eqref{eq:performance_metric_scalar} here corresponds to the area under the curve.}
	\label{fig:cost_Nordic_DAPI_droop}
\end{figure}

\subsection{General connected graphs with arbitrary susceptances}
\label{subsec:optimality_general_graphs}
In this section, we generalize the results in Section~\ref{subsec:optimality_complete} to power networks with general underlying graph structures. This implies that the eigenvalues $\lambda_i$ for $i\ge 2$ will in general not be equal, and minimizing $\norm{\mathcal{S}_\text{DAPI}}^2_{\text{sync}} $ is not equivalent to maximizing
$f_i(\gamma,q)$. The following theorem summarizes when there is a positive $\gamma^\star$, when it is zero, and when the optimum must be evaluated on a case-by-case basis.  

\begin{theorem}[Optimizing performance in general graphs]
\label{th_optimal_gamma_k_general}
Let Assumptions~\ref{ass:uniform_gains} and \ref{ass:communication_graph} hold. 
The optimal choices of $\gamma$ are as follows: \\
If $d \ge \sqrt{m\lambda_i}$ for all $i\ge 2,$ then \[ \gamma^\star=0, \] 
if $d \ge \sqrt{m\lambda_i}$ for $2 \le i \le l < n, $ then \[0 \le \gamma^\star  
\le \max_{i} \frac{ \sqrt{m\lambda_i } q -dq}{m\lambda_i},\]
and if $d < \sqrt{m\lambda_i}$ for all $i\ge 2$, then 
\[\min_{i} \frac{ \sqrt{m\lambda_i } q -dq}{m\lambda_i}  \le \gamma^\star \le \max_{i} \frac{ \sqrt{m\lambda_i } q -dq}{m\lambda_i}.\]
\end{theorem}

\begin{proof}
The case when $\displaystyle d \ge \sqrt{m\lambda_i}, \; i\ge 2$ is trivial since $f_i(\gamma,q)$ is then minimized by $\gamma^\star=0$ for all $i\ge 2$, by the proof of Theorem~\ref{th_optimal_gamma_k}. Now consider the case when $d < \sqrt{m\lambda_i}, \;  i \ge 2$. Assume for the sake of contradiction that $\gamma^\star < \min_{i} \frac{ \sqrt{m\lambda_i } q -dq}{m\lambda_i}$. This would imply that $\frac{\partial g_i(\gamma,q)}{\partial \gamma}\big|_{\gamma = \gamma^\star} < 0$ for $i\ge 2$, and hence $\gamma^\star$ cannot be the minimizer of $\norm{\mathcal{S}_\text{DAPI}}^2_{\text{sync}}$. The remaining cases are proven similarly. 
\end{proof}


\section{Simulations}
\label{sec:Simulations}
We consider a model of the Nordic power transmission network, consisting of $469$ generators. For simplicity we assume uniform generator inertia and damping, specifically $m=1 \times 10^4~\text{kgm}^2$ and $d =1 \times 10^4~\text{kgm}^2/\text{s}$. The admittances of the power transmission lines were assumed to be proportional to their inverse physical length. We simulate the response of the network after a random step disturbance on the load, taken from a normal distribution, with, respectively, the droop controller and the DAPI controller. 
For the DAPI controller, we set $q=2.5\times 10^{-7}$. 

The responses are shown in Fig.~\ref{fig:frequency_Nordic}, where two different choices of $\gamma$ have been used. In Fig.~\ref{fig:cost_Nordic_DAPI_droop}, the corresponding (squared) control error of the droop controlled and the DAPI controlled systems are compared for different choices of $\gamma$. It is clear that a smaller value of $\gamma$ substantially improves performance, as predicted by Theorem~\ref{th_optimal_gamma_k_general}. The same theorem predicts a large $\gamma$ to be the optimal choice as the inertia~$m$ and the gain~$q$ increase. In this case, however, DAPI offers only a marginal performance improvement, as a large $\gamma$ makes DAPI similar to CAPI control. 


\section{Discussion}
\label{sec:disc}
In this paper, we have studied the dynamic performance of synchronizing power networks by means of a coherency measure, which can be evaluated using an $\mathcal{H}_2$-norm framework. We showed that the performance of a power system controlled with DAPI control is always superior to that of a droop controlled as well as a CAPI controlled system for the same system and controller parameters. The performance of the respective controllers also scales differently with network size. Therefore, the advantage of DAPI control has the potential to become increasingly accentuated as the number of generators grows, for example, through large-scale integration of distributed generation sources. This is because the performance metric may grow unboundedly with CAPI or droop control, while it remains bounded with DAPI. We note that this result stands in contrast to previous work that has evaluated performance in terms or transient power losses~\cite{Bamieh2013Price,Tegling2015Price,Tegling2016_ACC}, where the losses, when normalized by the number of generators, always remained bounded. 


It was shown in Section~\ref{sec:optimal_DAPI} that the optimal choice of the communication gain $\gamma$ is zero in many cases, which would result in a decentralized proportional integral (DePI) controller. This, however, proves problematic as constant measurement noise can destabilize DePI controllers, unless phase measurement units are deployed at every node~\cite{andreasson2013lic}. Clearly, a positive $\gamma$, that ensures that the DAPI controller is robustly stable, must be chosen, although it may be suboptimal with respect to the performance metric considered here. A characterization of the performance of the DAPI controller under measurement noise is left as future work. 

\section*{Acknowledgements}
We would like to thank John W. Simpson-Porco for his insightful comments and a number of interesting discussions.

\bibliographystyle{IEEETran}
\bibliography{references}
\end{document}